\documentclass[12 pt]{article}
\usepackage{amssymb}
\usepackage{amsmath}
\usepackage{amsthm}
\allowdisplaybreaks
\usepackage{amscd}
\usepackage{graphicx}
\usepackage{hyperref}
\usepackage{geometry}
\usepackage{pxfonts}

\renewcommand{\theequation}{\arabic{section}.\arabic{equation}}
\def\vbar{\mathchoice{\vrule height6.3ptdepth-.5ptwidth.8pt\kern- .8pt}
{\vrule height6.3ptdepth-.5ptwidth.8pt\kern-.8pt} {\vrule
height4.1ptdepth-.35ptwidth.6pt\kern-.6pt} {\vrule
height3.1ptdepth-.25ptwidth.5pt\kern-.5pt}}
\def\<{\langle}
\def\>{\rangle}
\def\a{\alpha}

\def\c{\cdot}

\def\o{\otimes}

\def\f{\phi}

\def\r{\rho}

\def\pt{\prec}
\def\gr{\succ}
\def\o{\circ}

\newtheorem{df}{Definition}[section]
\newtheorem{thm}{Theorem}[section]
\newtheorem{cor}{Corollary}[section]
\newtheorem{rem}{Remark}[section]

\newtheorem{prop}{Proposition}[section]
\newtheorem{exa}{Example}[section]
\newtheorem{lem}{Lemma}[section]

\date{}
\begin{document}

\title{Construction of Hom-Pre-Jordan algebras \\and Hom-J-dendriform algebras}
\author{T. Chtioui, S. Mabrouk, A. Makhlouf}
\author{{ Taoufik Chtioui$^{1}$, Sami Mabrouk$^{2}$, Abdenacer Makhlouf$^{3}$ }\\
{\small 1.  University of Sfax, Faculty of Sciences Sfax,  BP
1171, 3038 Sfax, Tunisia} \\
{\small 2.  University of Gafsa, Faculty of Sciences Gafsa, 2112 Gafsa, Tunisia}\\
{\small 3.~ Universit\'e de Haute Alsace, IRIMAS - D\'epartement de Math\'ematiques, 
F-68093 Mulhouse, France}}
 \maketitle
\begin{abstract}
The aim of this work is to introduce and study the notions of Hom-pre-Jordan algebra
and Hom-J-dendriform algebra which generalize Hom-Jordan algebras. Hom-Pre-Jordan algebras are regarded as the underlying algebraic structures of the Hom-Jordan algebras behind the Rota-Baxter operators and  $\mathcal{O}$-operators  introduced in this paper. Hom-Pre-Jordan algebras are also analogues of Hom-pre-Lie algebras  for Hom-Jordan algebras. The anti-commutator of a Hom-pre-Jordan algebra is a Hom-Jordan algebra and the left multiplication operator gives a representation of a Hom-Jordan algebra.
On the other hand, a
Hom-J-dendriform algebra is a Hom-Jordan algebraic analogue of a Hom-dendriform algebra such that the anti-commutator of the sum of the two operations is a Hom-pre-Jordan algebra.
\end{abstract}

 \normalsize\vskip1cm

\section*{Introduction}
\renewcommand{\theequation}{\thesection.\arabic{equation}}

In order to study periodicity phenomena in algebraic K-theory,
J.-L. Loday introduced, in 1995,  the notion of dendriform algebra (see \cite{loday}).
Dendriform algebras are algebras with two
operations, which dichotomize the notion of associative algebra.
Later the notion of tridendriform algebra were
introduced by Loday and Ronco in their study of polytopes and Koszul duality (see \cite{tridend}).
In 2003 and in order to determine the algebraic structures behind a pair of commuting Rota-Baxter operators (on an associative algebra), Aguiar and Loday introduced the notion of quadri-algebra \cite{quadrialgebras}.  We refer to this kind of  algebras as Loday algebras. Thus, it is  natural to consider  the
Jordan algebraic analogues of Loday algebras as well as  their  Lie
algebraic analogues.

Jordan algebras were introduced in the context of
axiomatic quantum mechanics in 1932 by the physicist P. Jordan  and appeared in many areas of mathematics such as differential geometry, Lie theory, Physics  and analysis (see \cite{chu}, \cite{jacobson} and \cite{Upmeier} for more details).
The Jordan algebraic analogues of Loday algebras were considered.
Indeed, the notion of pre-Jordan algebra as a Jordan
algebraic analogue of a pre-Lie algebra was introduced in \cite{bai1}.
A pre-Jordan algebra is a vector space $A$ with a bilinear
multiplication $\c $ such that the product $x \circ y = x\c y+y \c x$ endows
$A$ with the structure of a Jordan algebra, and the left multiplication operator
$L(x) : y \mapsto x \c y$ define a representation of this Jordan algebra on $A$.
In other words, the product  $x \c y$ satisfies the following identities:
\begin{align*}
& (x \circ y)\c(z\c u)+ (y \circ z)\c(x\c u)+ (z \circ x)\c(y\c u)\\
&= z\c[(x \circ y) \c u]+ x\c[(y \circ z) \c u]+ y\c[(z \circ x) \c u],\\
& x \c [y \c (z \c u)] + z \c [y \c (x \c u)] + [(x \circ z) \circ y] \c u \\
&= z\c[(x \circ y) \c u]+ x\c[(y \circ z) \c u]+ y\c[(z \circ x) \c u]
\end{align*}
In order to find a dendriform algebra whose anti-commutator is a pre-Jordan
algebra,  Bai and Hou  introduced the notion of J-dendriform algebras  \cite{bai2}. They are, also related to pre-Jordan
algebras in the same way as pre-Jordan algebras are related to Jordan algebras.
They showed that an $\mathcal{O}$-operator (specially a Rota-Baxter operator of weight zero) on a pre-Jordan algebra or two commuting Rota-Baxter operators on a Jordan algebra give a J-dendriform algebra. In addition, They showed the relationships between J-dendriform algebras and Loday algebras especially quadri-algebras.

Hom-type algebras have been  investigated by many authors. In general,
Hom-type algebras are a kind of algebras in  which the  usual identities defining the structures are twisted by homomorphisms.
Such algebras appeared in the ninetieth in examples of $q$-deformations of the Witt and the Virasoro algebras. Motivated by these examples and their generalization, Hartwig, Larsson and Silvestrov introduced and studied  Hom-Lie algebras in \cite{hls}.

The notion of Hom-Jordan algebras was first introduced by A. Makhlouf in  \cite{mak1} with a connection to Hom-associative algebras and then   D. Yau modified slightly the definition in   \cite{yau1}  and established their relationships with Hom-alternative algebras.

We aim in this paper to introduce and study   Hom-pre-Jordan algebras and Hom-J-dendriform algebras generalizing, then, Hom-Jordan  algebras, pre-Jordan algebras and J-dendriform algebras.
The anti-commutator of a Hom-pre-Jordan algebra is a Hom-Jordan algebra and the left multiplication operators give a representation of this Hom-Jordan algebra, which is the beauty of such a structure.
Similarly, a Hom-J-dendriform algebra gives rise to a Hom-pre-Jordan algebra  and a Hom-Jordan algebra in the same way as Hom-dendriform algebra gives rise to Hom-pre-Lie algebra and Hom-Lie algebra (see \cite{makhloufDendriform}).

The paper is organized as follows. In Section 1, we recall some basic facts about Hom-Jordan algebras.   In Section 2, we introduce the notion of Hom-pre-Jordan algebra, provide  some properties and defien the notion of a bimodule of a Hom-pre-Jordan algebra. Moreover, we develop  some constructions theorems.
In Section 3, we introduce the notion of Hom-J-dendriform algebra and
study some of their fundamental properties  in terms of the
$\mathcal{O}$-operators of pre-Jordan algebras.

Throughout this paper $\mathbb{K}$ is a field of characteristic $0$ and all vector spaces are over $\mathbb{K}$.
We refer to a  Hom-algebra  as a tuple $(A,\mu,\alpha)$, where $A$ is a vector space, $\mu$ is a multiplication and $\alpha$ is a linear map. It is said to be regular if $\alpha$ is invertible. A Hom-associator is a trilinear map $as_{\alpha}$ defined for all $x,y,z\in A$ by
$as_{\alpha}(x,y,z)=(xy)\a(z)-\alpha(x) (yz)$.
When there is no ambiguity, we denote for simplicity the multiplication and composition by concatenation.

\section{Basic results on Hom-Jordan algebras}
In this section, we recall some basics about Hom-Jordan algebras introduced in \cite{yau1} and introduce the  notion of a representation of a Hom-Jordan algebra.
\begin{df}
A Hom-Jordan algebra is a Hom-algebra $(A,\circ ,\a)$  satisfying
the following conditions
\begin{align}\label{Hom jordan alg}
& x \circ y= y \circ x,\\
& as_\a (x \circ x,\a(y),\a(x))=0, \label{hom jordan identity}
\end{align}
for all $x,y\in A$.

\end{df}
 \begin{rem}
Since  the characteristic of $\mathbb{K}$ is $0$, condition \eqref{hom jordan identity} is equivalent to the following identity (for all $x,y,z,u \in A$)
 \begin{equation}\label{somme cyc}
 \circlearrowleft_{x,y,z} as_\a (x \circ y, \a(u), \a(z))=0,
 \end{equation}
  or equivalently,
  \begin{align}\label{somme cyc2}
& ((x \circ y) \circ \a(u)) \circ \a^2(z)  + ((y \circ z) \circ \a(u)) \circ \a^2(x) +((z \circ x) \circ \a(u)) \circ \a^2(y) \nonumber\\
&= \a(x \circ y) (\a(u) \circ \a(z)) + \a(y \circ z) (\a(u) \circ \a(x)) + \a(z \circ x) (\a(u) \circ \a(y)).
 \end{align}

 \end{rem}

 \begin{df}
Let  $(A, \circ, \a) $ be a Hom-Jordan algebra and $V$ be a vector space. Let $\r: A \to gl(V)$ be a linear map and $\f: V \to V$  be an algebra morphism. Then $(V,\r,\f)$ is called a representation (or a module) of $(A,\circ, \a)$, if for any $x,y,z \in A$
\begin{align}\label{representation}
& \f  \r(x)= \r(\a(x))  \f , \\
&\r(\a^2(x))\r(y\circ z)  \f + \r(\a^2(y))\r(z\circ x)  \f  + \r(\a^2(z))\r(x\circ y)  \f  \nonumber \\
&= \r(\a(x)\circ \a(y))\r(\a(z)) \f  +  \r(\a(y)\circ \a(z))\r(\a(x)) \f +  \r(\a(z)\circ \a(x))\r(\a(y)) \f, \\
& \r((x\circ y)\circ \a(z)) \f^2+ \r(\a^2(x))\r(\a(z))\r(y)+\r(\a^2(z))\r(\a(y))\r(x) \nonumber \\
&= \r(\a(x)\circ \a(y))\r(\a(z)) \f +  \r(\a(y)\circ \a(z))\r(\a(x)) \f +  \r(\a(z)\circ \a(x))\r(\a(y)) \f.
\end{align}
 \end{df}

\begin{prop}\label{semidirect product hom jordan}
Let $(A,\circ,\a)$ be a Hom-Jordan algebra, then $(V,\r,\f)$ is a representation  of $A$  if and only if there exists a Hom-Jordan algebra structure on the direct sum $A\oplus V$ (the semi-direct sum of the underlying vector spaces of $A$ and $V$) given by
\begin{align}\label{semidirect product}
& (x+u)\ast (y+v)= x\circ y + \r(x)v +\r(y)u, \quad \forall x,y \in A,\ u,v \in V.
\end{align}
We denote it by $A\ltimes_{\r,\f} V$ or simply $A\ltimes V$.
\end{prop}

 \begin{exa}
Let $(A,\circ,\a)$ be a Hom-Jordan algebra.  Let $ad :A \to gl(A)$  defined by  $ad(x)(y)=x\circ y=y\circ x$, for any $x,y \in A$.  Then $(A,ad,\a)$ is a representation of $(A,\circ,\a)$ called the adjoint  representation of $A$.
 \end{exa}

\begin{df}
Let $(A,\circ,\a)$ be a Hom-Jordan algebra and $(V,\r,\f)$ be representation.
A linear map $T: V \to A$ is called an $\mathcal{O}$-operator of $A$ associated to $\r$ if it satisfies
\begin{align}\label{O-operator of Jordan algebra}
 & T \f = \a T,\\
 & T(u) \circ T(v)=T \big( \r(T(u))v + \r(T(v))u \big),\quad \forall u,v \in V.
\end{align}
\end{df}
\begin{exa}
Let $(A,\circ,\a)$ be a Hom-Jordan algebra then a Rota-Baxter operator of weight zero is just an $\mathcal{O}$-operator on $A$ associated to the adjoint representation $(A,ad,\a)$.

\end{exa}

\section{Hom-Pre-Jordan algebras}
In this section, we generalize the notion of pre-Jordan algebra introduced in \cite{bai1} to the Hom case and study the relationships with Hom-Jordan algebras, Hom-dendriform algebras and Hom-pre-alternative algebras in terms of $\mathcal{O}$-operators of Hom-Jordan algebras.
\subsection{Definition and basic properties}
\begin{df}\label{hom prejordan}
A Hom-pre-Jordan algebra is a Hom-algebra  $(A, \c,\a)$  satisfying, for any $x,y,z,u \in A$, the following identities
\begin{align}\label{hom prejordan1}
& [\a(x) \circ \a(y)]\c[\a(z)\c \a(u)]+ [\a(y) \circ \a(z)]\c[\a(x)\c \a(u)]+ [\a(z) \circ \a(x)]\c[\a(y)\c \a(u)] \nonumber \\
&= \a^2(x)\c[(y \circ z) \c \a(u)]+ \a^2(y)\c[(z \circ x) \c \a(u)]+\a^2(z)\c[(x \circ y) \c \a(u)],\\
& [(x\circ z)\circ \a(y)]\c\a^2(u)+ \a^2(x)\c[\a(y)\c(z\c u)]+\a^2(z)\c[\a(y)\c(x\c u)] \nonumber\\
\label{hom prejordan2}&=\a^2(x)\c[(y \circ z) \c \a(u)]+ \a^2(y)\c[(z \circ x) \c \a(u)]+\a^2(z)\c[(x \circ y) \c \a(u)],
\end{align}
where $x \circ y= x \c y+ y \c x$.
When $\a$ is an algebra morphism,  the Hom-pre-Jordan algebra $(A, \c,\a)$ will be called multiplicative.
\end{df}
\begin{rem}
Eqs. \eqref{hom prejordan1} and \eqref{hom prejordan2} are equivalent to the following equations
$($for any $x,y,z,u\in A)$ respectively
\begin{align}\label{EqvHomPreJordan1}
 & (x,y,z,u)_\alpha^1+(y,z,x,u)_\alpha^1+(z,x,y,u)_\alpha^1   +(y,x,z,u)_\alpha^1+(x,z,y,u)_\alpha^1+(z,y,x,u)_\alpha^1=0,\\
&as_\alpha(\a(x),\a(y),z\c u)-as_\alpha(x\c z,\a(y),\a(u))+(y, z,x,u)_\alpha^2 +(y,x,z,u)_\alpha^2\nonumber\\&\label{EqvHomPreJordan2}+as_\alpha(\a(z),\a(y),x\c u)-as_\alpha(z\c x,\a(y),\a(u))=0,
\end{align}where
\begin{equation*}
  (x,y,z,u)_\alpha^1=[\a(x) \c  \a(y)]\c[\a(z)\c \a(u)]-\a^2(x)\c[(y \c z) \c \a(u)],
\end{equation*} \begin{equation*}
  (x, y,z,u)_\alpha^2  =[\a(x) \c\a(y)]\c[\a(z)\c \a(u)]-[\a(x)\c(y\c z)]\c\a^2(u),
\end{equation*}
\end{rem}
\begin{rem}
Any Hom-associative algebra is a Hom-pre-Jordan algebra.
\end{rem}

\begin{prop}
Let $(A,\c,\a)$ be a Hom-pre-Jordan algebra. Then the product given by
\begin{align}\label{prejordan=>jordan}
 x \circ y=x\c y+y\c x
\end{align}
defines a Hom-Jordan algebra structure  on $A$, which is called the associated Hom-Jordan algebra
of $(A,\c,\a)$ and $(A,\c,\a)$ is also called a compatible Hom-pre-Jordan algebra structure
on the Hom-Jordan algebra $(A, \circ,\a)$.
\end{prop}

\begin{proof}
Let $x,y,z,u\in A$,  it is easy to show that
 \begin{align*}
& ((x \circ y) \circ \a(u)) \circ \a^2(z)  + ((y \circ z) \circ \a(u)) \circ \a^2(x) +((z \circ x) \circ \a(u)) \circ \a^2(y) \nonumber\\
&= (\a(x) \circ \a(y)) (\a(u) \circ \a(z)) + (\a(y) \circ \a(z)) (\a(u) \circ \a(x)) +( \a(z )\circ\a( x)) (\a(u) \circ \a(y)).
 \end{align*}
if and only if $l_1+l_2+l_3+l_4=r_1+r_2+r_3+r_4$ where
\begin{align*}
  l_1 &=\circlearrowleft_{x,y,z}\a^2(x)\c[(y \circ z) \c \a(u)], \\
  l_2 & =[(x\circ y)\circ \a(u)]\c\a^2(z)+ \a^2(x)\c[\a(u)\c(y\c z)]+\a^2(y)\c[\a(u)\c(x\c z)] ,\\
  l_3 & =[(x\circ z)\circ \a(u)]\c\a^2(y)+ \a^2(x)\c[\a(u)\c(z\c y)]+\a^2(z)\c[\a(u)\c(x\c u)] ,\\
  l_4 & =[(y\circ z)\circ \a(u)]\c\a^2(x)+ \a^2(y)\c[\a(u)\c(z\c x)]+\a^2(z)\c[\a(u)\c(y\c x)] ,
\end{align*}

and
\begin{align*}
  r_1&=\circlearrowleft_{x,y,z}(\a(x)\circ\a(y))\cdot(\a(z)\cdot\a(u)),  \\
   r_2&=\circlearrowleft_{x,y,u}(\a(x)\circ\a(y))\cdot(\a(u)\cdot\a(z)),  \\
   r_3&=\circlearrowleft_{x,z,u}(\a(x)\circ\a(z))\cdot(\a(u)\cdot\a(y)),\\
   r_4&=\circlearrowleft_{y,z,u}(\a(y)\circ\a(z))\cdot(\a(u)\cdot\a(x)).
\end{align*}
Now using Definition \ref{hom prejordan}, we can easily see that $l_i=r_i$, for $i=1,...,4$.
\end{proof}
The following conclusion can be obtained straightforwardly using the previous proposition.
\begin{prop}
Let $(A,\c,\a)$ be a Hom-algebra.
Then $(A,\c,\a)$ is a Hom-pre-Jordan algebra if and only if $(A, \circ,\a)$ defined by Eq. \eqref{prejordan=>jordan}
is a Hom-Jordan algebra and $(A,L,\a)$ is a representation of $(A, \circ,\a)$, where $L$ denotes the left multiplication operator on $A$.
\end{prop}

\begin{proof}straightforward.
\end{proof}

\begin{prop} \label{Homjordan=>Homprejordan}
Let $(A, \circ,\a)$ be a Hom-Jordan algebra and $(V,\r,\f)$ be a representation.
If $T$ is an $\mathcal{O}$-operator associated to $\r$, then $(V,\ast,\f)$ is a Hom-pre-Jordan
algebra, where\begin{align}\label{u etoile v}
    u \ast v= \r(T(u))v, \quad \forall u,v \in V.
\end{align}
Therefore there exists an associated Hom-Jordan algebra structure on $V$ given by
Eq. \eqref{prejordan=>jordan} and $T$ is a homomorphism of Hom-Jordan algebras. Moreover, $T (V) =\{T (v) | v \in V\}\subset A$ is a Hom-Jordan subalgebra of $(A, \circ,\a)$ and there is an induced
Hom-pre-Jordan algebra structure on $T(V)$ given by
\begin{align}
T (u).T (v) = T (u \ast v), \quad \forall u,v \in V.
\end{align}
The corresponding associated Hom-Jordan algebra structure on $T(V)$ given by
Eq. \eqref{prejordan=>jordan} is just a Hom-Jordan subalgebra of $(A, \circ,\a)$ and $T$ is a homomorphism of Hom-pre-Jordan algebras.
\end{prop}

\begin{proof}
Let $u,v,w,a\in V$ and put $x=T(u)$, $y=T(v)$, $z=T(w)$ and $u\bullet v=u\ast v+v\ast u$. Note first that $T(u\bullet v)=T(u)\circ T(v)$. Then
\begin{align*}
  (\f(u)\bullet \f(v))\ast(\f(w)\ast\f(a)) &=\rho(T(\rho(T(\f(u)\bullet\f(v)))))\rho(T(\f(w)))\f(a)  \\
   &=\rho(T(\f(u))\circ T(\f(v)))\rho(T(\f(w)))\f(a)\\
   &=\rho(\a(x)\circ \a(y))\rho(\a(z))\f(a),
\end{align*}
\begin{align*}
 \f^2(u)\ast[(v\bullet w)\ast \f(a)]   & = \rho(T(\f^2(u)))\rho(T(v\bullet w) ) \f(a) \\
   &= \rho(T(\f^2(u)))\rho(T(v)\circ T(w)) \f(a)\\
   &=\rho(\a^2(x))\rho(y\circ z) \f(a),
\end{align*}
\begin{align*}
  \f^2(u)\ast[\f(v)\ast(w\ast a)] & =\rho(T(\f^2(u)))\rho(T(\f(v)))\rho(T(w)) a \\
    & =\rho(\a^2(x))\rho(\a(y))\rho(z) a,
\end{align*}
\begin{align*}
  [(u\bullet v)\bullet \f(w)]\ast \f^2(a) & = \rho(T([(u\bullet v)\bullet \f(w)])) \f^2(a) \\
   & =\rho([T(u\bullet v)\circ T(\f(w))]) \f^2(a) \\
   & =\rho([(T(u)\circ T(v))\circ T(\f(w))]) \f^2(a)\\
   &=\rho([(x\circ y)\circ \a(z)]) \f^2(a).
\end{align*}
Hence,
\begin{align*}
   & (\f(u)\bullet \f(v))\ast(\f(w)\ast\f(a)) +(\f(v)\bullet \f(w))\ast(\f(u)\ast\f(a)) +(\f(w)\bullet \f(u))\ast(\f(v)\ast\f(a))  \\
   & =\rho(\a(x)\circ \a(y))\rho(\a(z))\f(a)+\rho(\a(y)\circ \a(z))\rho(\a(x))\f(a)+\rho(\a(z)\circ \a(x))\rho(\a(y))\f(a) \\
   & =\rho(\a^2(x))\rho(y\circ z) \f(a)+\rho(\a^2(y))\rho(z\circ x) \f(a)+\rho(\a^2(z))\rho(x\circ y) \f(a)\\
   &+\f^2(u)\ast[(v\bullet w)\ast \f(a)]+\f^2(v)\ast[(w\bullet u)\ast \f(a)]+\f^2(w)\ast[(u\bullet v)\ast \f(a)],
\end{align*}
and
\begin{align*}
   & [(u\bullet v)\bullet \f(w)]\ast \f^2(a)+\f^2(u)\ast[\f(w)\ast(v\ast a)]+\f^2(w)\ast[\f(v)\ast(u\ast a)] \\
   & =\rho([(x\circ y)\circ \a(z)]) \f^2(a)+\rho(\a^2(x))\rho(\a(z))\rho(y) a+\rho(\a^2(z))\rho(\a(y))\rho(x) a \\
   & =\rho(\a^2(x))\rho(y\circ z) \f(a)+\rho(\a^2(y))\rho(z\circ x) \f(a)+\rho(\a^2(z))\rho(x\circ y) \f(a)\\
   &+\f^2(u)\ast[(v\bullet w)\ast \f(a)]+\f^2(v)\ast[(w\bullet u)\ast \f(a)]+\f^2(w)\ast[(u\bullet v)\ast \f(a)].
\end{align*}
Therefore, $(V,\ast,\f)$ is a Hom-pre-Jordan algebra. The other conclusions follow immediately.
\end{proof}

An obvious consequence of Proposition \ref{Homjordan=>Homprejordan} is the following construction of a Hom-pre-Jordan algebra in terms of a Rota-Baxter operator $($of weight zero$)$ of a Hom-Jordan algebra.

\begin{cor}
Let $(A, \circ,\a)$ be a Hom-Jordan algebra and $R$ be a Rota-Baxter
operator $($of weight zero$)$ on $A$. Then there is a Hom-pre-Jordan algebra structure
on $A$ given by
$$x\c y= R(x) \circ y, \quad \forall x,y \in A.$$
\end{cor}

\begin{proof}
Straightforward.
\end{proof}

\begin{exa}
Let $\{e_1,e_2,e_3\}$ be a basis of   a $3$-dimensional vector space $A$ over $\mathbb{K}$. The following product $\circ $ and the linear map $\a$ define Hom-Jordan algebras over $\mathbb{K}$. 
$$
\begin{array}{c|c|c|c}
  \circ  & e_1 & e_2 & e_3 \\
  \hline
  e_1& ae_1 & ae_2 & be_3 \\
  \hline
  e_2 & ae_2 & ae_2 & \frac{b}{2}e_3 \\
  \hline
  e_3 & be_3 & \frac{b}{2}e_3 & 0 
\end{array},$$
\begin{align*}
    \a(e_1)=ae_1,\ \a(e_2)=ae_2,\ \a(e_3)=be_3,
\end{align*}
where $a$ and $b$ are parameters in $\mathbb{K}$.  Let $R$ be the operator defined with respect to the basis $\{e_1,e_2,e_3\}$ by 
\begin{align*}
    R(e_1)=\lambda_1 e_3,\ R(e_2)=\lambda_2 e_3,\  R(e_3)=0,
\end{align*}
where $\lambda_1$ and $\lambda_2$ are parameters in $\mathbb{K}$.  Then we can easily check that $R$ is a Rota-Baxter operator on $A$.  
Now, using  the previous corollary,   there is a Hom-pre-Jordan algebra structure
on $A$, with the same twist map and the multiplication given by
$x\c y= R(x) \circ y, \quad \forall x,y \in A$, that is 

$$
\begin{array}{c|c|c|c}
  \cdot  & e_1 & e_2 & e_3 \\
  \hline
  e_1& \lambda_1 b e_3 & \lambda_1 \frac{b}{2}e_3 & 0 \\
  \hline
  e_2 & \lambda_2 b e_3 & \lambda_2 \frac{b}{2}e_3 & 0\\
  \hline
  e_3 & 0 & 0  & 0 
\end{array}.$$
 
\end{exa}

\begin{cor}
Let $(A,\circ,\a)$ be a Hom-Jordan algebra. Then there exists a compatible
Hom-pre-Jordan algebra structure on $A$ if and only if there exists an invertible
$\mathcal{O}$-operator of $(A,\circ,\a)$.
\end{cor}

\begin{proof}Let $(A,\c,\a)$ be a Hom-pre-Jordan algebra and $(A,\circ,\a)$ be the associated
Hom-Jordan algebra. Then the identity map $id : A \to A$ is an invertible $\mathcal{O}$-operator of $(A,\circ,\a)$
associated to $(A,ad,\a)$.

 Conversely, suppose that there exists an invertible $\mathcal{O}$-operator $T$ of $(A,\circ,\a)$ associated to a
representation $(V,\rho,\f)$, then by Proposition \ref{Homjordan=>Homprejordan}, there is a Hom-pre-Jordan algebra
structure on $T(V)=A$ given by
$$T(u)\c T(v)=T(\rho(T(u))v),\ \textrm{for\ all}\ u,v\in V.$$
If we set $T(u)=x$ and $T(v)=y$, then we obtain
$$x\c y=T(\rho(x)T^{-1}(y)),\ \textrm{for\ all}\ x,y\in A.$$
It is a compatible Hom-pre-Jordan algebra structure on $(A,\circ,\a)$. Indeed,
\begin{align*}
  x\c y+y \c x & =T(\rho(x)T^{-1}(y)+\rho(y)T^{-1}(x)) \\
   &= T(T^{-1}(x))\circ T(T^{-1}(y))=x\circ y.
\end{align*}
\end{proof}
The following result reveals the relationship between Hom-pre-Jordan algebras, Hom-pre-alternative algebras  and so Hom-dendriform algebras. we recall the following definitions introduced in \cite{homprealt, makhloufDendriform}
\begin{df}
  A Hom-pre-alternative algebra is a quadruple $(A,\prec,\succ,\a)$ where $\prec,\succ:A\otimes A\to A$  and   $\a:A\to A$ are
 linear maps satisfying
  \begin{align}
 &(x\succ y)\prec \a(z)- a(x)\succ( y\prec z)+(y\prec x)\prec \a(z)- a(y)\prec( x\star z)=0 ,\label{HomPreAlt1} \\
& (x\succ y)\prec \a(z)- a(x)\succ( y\prec z)+(x\star z)\succ \a(y)- a(x)\succ( z\succ y)=0 ,\label{HomPreAlt2}\\
 & (x\prec y)\prec \a(z)- a(x)\prec( y\star z)+(x\prec z)\prec \a(y)- a(x)\prec( z\star y)=0,\label{HomPreAlt3}\\
&(x\star y)\succ \a(z)- a(x)\succ( y\succ z)+(y\star x)\succ \a(z)- a(y)\succ( x\succ z)=0,\label{HomPreAlt4}
 \end{align}
 for all $x,y,z\in A$, where $x\star y=x \prec y+ x \succ y$.
\end{df}

\begin{df}
  A Hom-Dendriform algebra is a quadruple $(A,\prec,\succ,\a)$ where $\prec,\succ:A\otimes A\to A$  and   $\a:A\to A$ are
 linear maps satisfying
 \begin{align}
&(x\succ y)\prec \a(z)- a(x)\succ( y\prec z)=0 ,\label{HomDedriform1} \\
& (x\prec y)\prec \a(z)- a(x)\prec( y\star z)=0,\label{HomDendriform2}\\
&(x\star y)\succ \a(z)- a(x)\succ( y\succ z)=0,\label{HomDendriform3}
\end{align}
 for all $x,y,z\in A$, where $x\star y=x \prec y+ x \succ y$.
\end{df}

\begin{prop}
Let $(A,\prec,\succ,\a)$ be a Hom-pre-alternative algebra. Then the product given by
\begin{align*}
    x\c y=x \succ y+ y \prec x, \quad \forall x, y \in A,
\end{align*}
defines a Hom-pre-Jordan algebra structure on $A$.
\end{prop}

\begin{proof}
Let $x,y,z, u\in A$, set $x\star y=x \prec y+ x \succ y$ and $x\circ y=x \c y+ y \c x=x \star y+ y \star x $. We will just prove the identity \eqref{hom prejordan1}.  One has
\begin{align*}
& \circlearrowleft_{x,y,z}\Big([\a(x) \circ \a(y)]\c[\a(z)\c \a(u)]-\a^2(x)\c[(y \circ z) \c \a(u)]\Big)\\
&=\circlearrowleft_{x,y,z}\Big([\a(x) \circ \a(y)]\gr[\a(z)\gr \a(u)]+[\a(x) \circ \a(y)]\gr[\a(u)\pt \a(z)]\\&+[\a(z) \gr \a(y)]\pt[\a(x)\circ \a(y)]+[\a(u) \pt \a(z)]\pt[\a(x)\circ \a(y)]\\
&-\a^2(x)\gr[(y \circ z) \gr \a(u)]-\a^2(x)\gr[\a(u)\pt(y \circ z) ]\\&-[(y \circ z) \gr \a(u)]\pt\a^2(x)-[\a(u)\pt(y \circ z)]\pt\a^2(x)\Big)
\\
&=\circlearrowleft_{x,y,z}\Big([(x \circ y)\circ\a(z)]\gr \a^2(u)+[\a(u) \pt \a(z)]\pt[\a(x)\circ \a(y)]-[\a(u)\pt(y \circ z)]\pt\a^2(x)\Big).
\end{align*}
Since $(A,\star,\a)$ is a Hom-alternative algebra (see \cite{homprealt}), then we have $$\circlearrowleft_{x,y,z}\Big([(x \circ y)\circ\a(z)]\gr \a^2(u)\Big)=0.$$
In addition  using the fact that $(A,\pt,\gr,\a)$ is a Hom-pre-alternative algebra,  then we obtain
$$\circlearrowleft_{x,y,z}\Big([(x \circ y)\circ\a(z)]\gr \a^2(u)+[\a(u) \pt \a(z)]\pt[\a(x)\circ \a(y)]-[\a(u)\pt(y \circ z)]\pt\a^2(x)\Big)=0.$$
The identity \eqref{hom prejordan2} can be done by a similar treatment.
\end{proof}
 Since any  Hom-dendriform algebra is a Hom-pre-alternative algebra, we obtain the following conclusion.
\begin{cor}
Let $(A,\prec,\succ,\a)$ be a Hom-dendriform algebra. Then the product given by
\begin{align*}
    x\c y=x \succ y+ y \prec x, \quad \forall x, y \in A,
\end{align*}
defines a Hom-pre-Jordan algebra structure on $A$.
\end{cor}
\subsection{Bimodules and $\mathcal{O}$-operators}
 In this section, we introduce and study bimodules of Hom-pre-Jordan algebras.
 \begin{df}
 Let $(A,\c,\a)$ be a Hom-pre-Jordan algebra and $V$ be a vector space. Let  $l,r: A \to gl(V)$ be two linear maps and $\f \in gl(V)$. Then $(V,l,r,\f)$ is called a bimodule of $A$ if the following conditions hold $($for any $x,y,z \in A)$
\begin{align}
& \f l(x)=l(\a(x)) \f,\ \f r(x)=r(\a(x)) \f  , \label{rephomprejor1} \\
&l(\a^2(x))l(y\circ z)  \f + l(\a^2(y))l(z\circ x)  \f  + l(\a^2(z))l(x\circ y)  \f  \nonumber \\
&= l(\a(x)\circ \a(y))l(\a(z)) \f  +  l(\a(y)\circ \a(z))l(\a(x)) \f +  l(\a(z)\circ \a(x))l(\a(y)) \f, \label{rephomprejor2}\\
& l((x\circ z)\circ \a(y)) \f^2+ l(\a^2(x))l(\a(z))l(y)+l(\a^2(z))l(\a(y))l(x) \nonumber \\
&= l(\a(x)\circ \a(y))l(\a(z)) \f +  l(\a(y)\circ \a(z))l(\a(x)) \f +  l(\a(z)\circ \a(x))l(\a(y)) \f ,\label{rephomprejor3}\\
& \f\big(l(x \circ y)r(z) + r(x · z)l(y) + r(y · z)r(x) + r(x · z)r(y) + r(y · z)l(x)\big) \nonumber\\
& = l(\a^2(x))r(\a(z))l(y) + l(\a^2(y))r(\a(z))r(x) + r[(x \circ y)\a(z)]\f^2\nonumber\\
&+l(\a^2(y))r(\a(z))l(x)+l(\a^2(x))r(\a(z))r(y), \label{rephomprejor4}\\
&\f\big(r(z · y)l(x) + r(x · y)r(z) + l(x \circ z)r(y) + r(x · y)l(z) + r(z · y)r(x)\big)\nonumber\\
&=\big( l(\a^2(x))r(z · y) + r(\a^2(y))r(x \circ z)
 + r(\a^2(y))l(x \circ z) + l(\a^2(z))r(x · y)\big)\f ,\label{rephomprejor5}\\
& \f\big(l(x · y)r(z) + r(x · z)l(y) + r(y · z)r(x) 
  + l(y · x)r(z) + r(x · z)r(y) + r(y · z)l(x) \big) \nonumber\\
&= l(\a^2(x))l(\a(y))r(z) + r(\a^2(z))l(\a(y))r(x) + r(\a^2(z))r(\a(y))r(x) \nonumber \\
& +r(\a^2(z))l(\a(y))l(x) + r[\a(y) · (x · z)]\f^2 + r(\a^2(z))r(\a(y))l(x), \label{rephomprejor6}
\end{align}
where $x \circ y= x \c y+ y \c x$.
\end{df}

\begin{prop}\label{semidirectproduct hompreJordan}
Let $(A,\c,\a)$ be a Hom-pre-Jordan algebra, $V$ be a vector space, $l,r:A\to gl(V)$ be linear maps and $\f \in gl(V)$. Then $(V,l,r,\f)$ is a   bimodule of $A$ if and only if
the direct sum $A \oplus V$ (as vector space) turns into a Hom-pre-Jordan algebra (the
semidirect sum) by defining the multiplication in $A\oplus V$ as
$$(x+u)\ast (y+v)=x \c y+ l(x)v+ r(y)u,\ \quad \forall x,y \in A,\ u,v \in V.$$
We denote it by $A\ltimes_{l,r}^{\a,\f} V$ or simply $A\ltimes V$.
\end{prop}
\begin{prop}
  Let $(V,l,r,\f)$ be a   bimodule of a Hom-pre-Jordan algebra $(A,\c,\a)$  and $(A,\circ,\a)$ be its associated Hom-Jordan algebra. Then
  \begin{itemize}
    \item [(a)] $(V,l,\f)$ is a representation of $(A,\circ,\a)$,
    \item [(b)] $(V,l+r,\f)$ is a representation of $(A,\circ,\a)$.
  \end{itemize}
\end{prop}
\begin{proof}
(a) follows immediately from Eqs. \eqref{rephomprejor2}-\eqref{rephomprejor3}.
For (b),  by Proposition \ref{semidirectproduct hompreJordan},   $A\ltimes_{l,r}^{\a,\f} V$ is a Hom-pre-Jordan algebra. Consider its associated Hom-Jordan algebra $(A \oplus V, \widetilde{\o}, \a+\f)$, we have
\begin{align*}
& (x+u) \widetilde{\o} (y+v)=(x+u)\ast (y+v) + (y+v)\ast (x+u) \\
&= x \c y+ l(x)v+ r(y)u + y \c x + l(y)u+ r(x)v \\
&= x \o y + (l+r)(x)v + (l+r)(y)u .
\end{align*}
According to Proposition \ref{semidirect product hom jordan}, we deduce that  $(V,l+r,\f)$ is a representation of $(A,\circ,\a)$.
\end{proof}
\begin{df}
Let $(A,\c,\a)$ be a Hom-pre-Jordan algebra and $(V,l,r,\f)$ be a bimodule.  A linear map $T:V \to A$ is called an $\mathcal{O}$-operator of  $(A,\c,\a)$ associated to  $(V,l,r,\f)$ if
\begin{align}\label{O-operator hom preJordan}
& T \f= \a T,\\
& T(u).T(v)=T\big(l(T(u))v+ r(T(v))u \big), \quad \forall u,v \in V.
\end{align}
\end{df}
\begin{rem}
  If $T$ is  an $\mathcal{O}$-operator of a Hom-pre-Jordan algebra  $(A,\c,\a)$ associated to  $(V,l,r,\f)$, then $T$ is  an $\mathcal{O}$-operator of its associated  Hom-Jordan algebra  $(A,\circ,\a)$ associated to  $(V,l+r,\f)$.
\end{rem}
\section{Hom-J-dendriform algebras}
\begin{df}
A Hom-J-dendriform algebra is a quadruple $(A,\prec,\succ,\a)$, where $A$ is a vector space equipped with a linear map  $\a: A \to A$ and  two products denoted by $\pt, \gr: A \otimes A \to A$ satisfying the following identities $($for any $x,y,z,u \in A)$
\begin{align}
& \a(x \o y) \gr \a(z \gr u)+  \a(y \o z) \gr \a(x \gr u)+  \a(z \o x) \gr \a(y \gr u) \nonumber \\
\label{Hom-J-dend alg1}&= \a^2(x) \gr [(y\o z) \gr \a(u)]+ \a^2(y) \gr [(z\o x) \gr \a(u)]+ \a^2(z) \gr [(x\o y) \gr \a(u)],\\
& \a(x \o y) \gr \a(z \gr u)+  \a(y \o z) \gr \a(x \gr u)+  \a(z \o x) \gr \a(y \gr u) \nonumber \\
\label{Hom-J-dend alg2}&=\a^2(x) \gr[\a(y) \gr (z \gr u)] +\a^2(z) \gr[\a(y) \gr (x \gr u)] +[\a(y) \o (z \o x)]\gr \a^2(u),\\
&\a(x \o y)\gr \a(z \pt u)+\a(x \c z)\pt \a(y\diamond u)+ \a(y\c z)\pt \a(x \diamond u) \nonumber\\
\label{Hom-J-dend alg3}&=\a^2(x)\gr [\a(z)\pt (y \diamond u)]+ \a^2(y)\gr [\a(z)\pt (x \diamond u)]+[(x \o y)\c \a(z)]\pt\a^2(u),\\
& \a(z \c y)\pt \a(x\diamond u)+ \a(x \c y)\pt \a(z\diamond u)+ \a(x \o z)\gr \a(y \pt u) \nonumber\\
\label{Hom-J-dend alg4}&=\a^2(x)\gr [(z \c y) \pt \a(u)]+\a^2(z)\gr [(x \c y) \pt \a(u)]+\a^2(y)\pt [(x\o z)\diamond \a(u)],\\
&\a(x\o y)\gr \a(z \pt u)+\a(x\c z)\gr \a(y \diamond u)+\a(y \c z)\pt \a(x \diamond u)\nonumber\\
\label{Hom-J-dend alg5}&=\a^2(x)\gr [\a(y)\gr (z \pt u)]+ \a^2(z)\pt [\a(y)\diamond (z \diamond u)]+[\a(y)\c (x \c z)]\pt \a^2(u),
\end{align}
where
\begin{align}
&x\c y=x \gr y+ y \pt x,   \label{produit point} \\
& x \diamond y= x \gr y +x \pt y, \label{produit diamond} \\
& x\o y=x\c y+ y \c x=x\diamond y+y\diamond x. \label{produit rond}
\end{align}
\end{df}
\begin{rem}
  Let $(A,\pt,\gr,\a)$ be a Hom-J-dendriform algebra, if $\pt:=0$ then $(A,\gr,\a)$ is a Hom-pre-Jordan algebra.
\end{rem}

\begin{prop}
Let $(A,\pt,\gr,\a)$ be a Hom-J-dendriform algebra.
\begin{itemize}
\item [(a)] The product given by Eq. \eqref{produit point} defines a Hom-pre-Jordan algebra $(A,\c,\a)$, called the associated vertical Hom-pre-Jordan algebra.
\item [(b)]The product given by Eq. \eqref{produit diamond} defines a Hom-pre-Jordan algebra $(A,\diamond,\a)$, called the associated horizontal Hom-pre-Jordan algebra.
\item [(c)] The associated vertical and horizontal Hom-pre-Jordan algebras $(A,\c,\a)$ and $(A,\diamond,\a)$ have the same associated Hom-Jordan algebra $(A,\o,\a)$ defined by Eq. \eqref{produit rond}, called the associated Hom-Jordan algebra of $(A,\pt,\gr,\a)$.
\end{itemize}
\end{prop}

\begin{proof}
We will juste prove (a), Let $x,y,z,u\in A$
\begin{align*}
   &[\a(x) \circ \a(y)]\c[\a(z)\c \a(u)]+ [\a(y) \circ \a(z)]\c[\a(x)\c \a(u)]+ [\a(z) \circ \a(x)]\c[\a(y)\c \a(u)]   \\
   &=[\a(x) \circ \a(y)]\gr[\a(z)\gr \a(u)]+ [\a(y) \circ \a(z)]\gr[\a(x)\gr \a(u)]+ [\a(z) \circ \a(x)]\gr[\a(y)\gr \a(u)]\\
   &+[\a(x) \circ \a(y)]\gr[\a(u)\pt \a(z)]+ [\a(x) \c \a(u)]\pt[\a(y)\diamond \a(z)]+ [\a(y) \c \a(u)]\pt[\a(x)\diamond \a(z)]\\
   &+[\a(y) \circ \a(z)]\gr[\a(u)\pt \a(x)]+ [\a(z) \c \a(u)]\pt[\a(y)\diamond \a(x)]+ [\a(y) \c \a(u)]\pt[\a(z)\diamond \a(x)]\\
   &+[\a(z) \circ \a(x)]\gr[\a(u)\pt \a(y)]+ [\a(x) \c \a(u)]\pt[\a(z)\diamond \a(y)]+ [\a(z) \c \a(u)]\pt[\a(x)\diamond \a(y)]\\
&=\a^2(x)\gr[(y \circ z) \gr \a(u)]+ \a^2(y)\gr[(z \circ x) \gr \a(u)]+\a^2(z)\gr[(x \circ y) \gr \a(u)]\\
   &+\a^2(x)\gr[\a(u)\pt(y \diamond z) ]+ \a^2(y)\gr[ \a(u)\pt(x \diamond z)]+[(x \circ y) \c \a(u)]\pt\a^2(x)\\
   &+\a^2(z)\gr[\a(u)\pt(y \diamond x) ]+ \a^2(y)\gr[\a(u)\pt(z \diamond x) ]+[(z \circ y) \c \a(u)]\pt\a^2(x)\\
   &+\a^2(z)\gr[\a(u)\pt(x \diamond y) ]+ \a^2(x)\gr[\a(u)\pt(z \diamond y) ]+[(z \circ x) \c \a(u)]\pt\a^2(y)\\
&=\a^2(x)\c[(y \circ z) \c \a(u)]+ \a^2(y)\c[(z \circ x) \c \a(u)]+\a^2(z)\c[(x \circ y) \c \a(u)].
\end{align*}
Similarly, we can get \eqref{hom prejordan2}.
\end{proof}

\begin{prop}
Let $(A,\pt,\gr,\a)$ be a Hom-J-dendriform algebra. Then $(A,L_{\gr},R_{\pt},\a)$ is a bimodule of  its  associated horizontal Hom-pre-Jordan algebra $(A,\diamond,\a)$.
\end{prop}

\begin{proof} We  check  Eq. \eqref{rephomprejor2} and Eq. \eqref{rephomprejor5}.
In fact, for any $x,y,z,u \in A$, we have
\begin{align*}
& L_{\gr}(\a^2(x))L_{\gr}(y\circ z)  \a(u) + L_{\gr}(\a^2(y))L_{\gr}(z\circ x)  \a(u)  + L_{\gr}(\a^2(z))L_{\gr}(x\circ y)  \a(u) \\
&=\a^2(x)\gr [(y\circ z) \gr \a(u) ]+\a^2(y) \gr[(z\circ x) \gr \a(u)]  + \a^2(z) \gr[(x\circ y) \gr \a(u) ] \\
&=  \a(x \o y) \gr \a(z \gr u)+  \a(y \o z) \gr \a(x \gr u)+  \a(z \o x) \gr \a(y \gr u)  \\
&= L_{\gr} (\a(x) \o \a(y)) L_{\gr}(\a(z))\a(u) +   L_{\gr} (\a(y) \o \a(z)) L_{\gr}(\a(x))\a(u) +  L_{\gr} (\a(z) \o \a(x)) L_{\gr}(\a(y))\a(u) .
\end{align*}
Moreover,
\begin{align*}
& \a\big(R_{\pt}(z \diamond y)L_{\gr}(x) u+ R_{\pt}(x \diamond y)R_{\pt}(z)u +
 L_{\gr}(x \circ z)R_{\pt}(y) u+ R_{\pt}(x \diamond y)L_{\gr}(z) u+ R_{\pt}(z \diamond y)R_{\pt}(x)u\big) \\
&= \a(x \gr u) \pt \a(z \diamond y) + \a(u \pt z) \pt \a(x \diamond y) + \a(x \o z)\gr \a(u \pt y)   \\
&+  \a(z \gr u) \pt \a(x \diamond y) + \a(u \pt x)\pt \a(z \diamond y) \\
&=  \a(x \c u) \pt \a(z \diamond y)  +  \a(z \c u) \pt \a(x \diamond y) + \a(x \o z)\gr \a(u \pt y) \\
&= \a^2(x)\gr [\a(u)\pt (z \diamond y)]+ \a^2(z)\gr [\a(u)\pt (x \diamond y)]+[(x \o z)\c \a(u)]\pt\a^2(y)\\
&= L_{\gr}(\a^2(x)) R_{\pt}(z \diamond y)\a(u) +  L_{\gr}(\a^2(z)) R_{\pt}(x \diamond y)\a(u) \\
 &+ R_{\pt}(\a^2(y))L_{\gr}(x \o z)\a(u) +  R_{\pt}(\a^2(y))R_{\pt}(x \o z) \a(u) .
\end{align*}
Other identities can be proved using similar computations.
\end{proof}

\begin{prop}
Let $(A,\pt,\gr)$ be a J-dendriform algebra  and $\a:A \to A$ be an algebra morphism. Then $(A,\pt_\a,\gr_\a,\a)$ is a Hom-J-dendriform algebra, where for any $x,y \in A$
$$x \pt_\a y= \a(x) \pt \a(y),\ \ \ x \gr_\a y= \a(x) \gr \a(y).$$
\end{prop}

\begin{proof}
Straightforward.
\end{proof}

\begin{prop}
Let $(A,\pt,\gr,\a)$ be a Hom-J-dendriform algebra. Define two bilinear products $\pt^t, \gr^t$ on $A$  by
\begin{align}
    x \pt^t y =y \pt x, \ x \gr^t y =y \gr x, \quad \forall x,y \in A.
\end{align}
Then $(A, \pt^t,\gr^t,\a)$ is a Hom-J-dendriform algebra called the transpose of $A$.  Moreover, its associated horizontal Hom-pre-Jordan algebra is the associated vertical Hom-pre-Jordan algebra $(A, \c,\a)$ of $(A, \pt,\gr,\a)$
and its associated vertical Hom-pre-Jordan algebra is the associated horizontal
Hom-pre-Jordan algebra $(A,\diamond,\a)$ of $(A,\pt,\gr,\a)$.
\end{prop}

\begin{proof}
Note first that \begin{align*}
                 & x\c^t y =x\gr^ty+y\pt^tx=x\gr y+x\pt y=x\diamond y,\\
                    & x\diamond^t y =x\gr^ty+x\pt^ty=x\gr y+y\pt x=x\c y,\\
                    &x\circ^t y =x\gr^ty+x\pt^ty+y\gr^t x+y\pt^t x=x\gr y+y\pt x+y\gr x+x\pt y=x\circ y.
                \end{align*}
                Therefore  we can easily check  (Eq. \eqref{Hom-J-dend alg1})$^t$=(Eq. \eqref{Hom-J-dend alg1}), (Eq. \eqref{Hom-J-dend alg2})$^t$=(Eq. \eqref{Hom-J-dend alg2}), (Eq. \eqref{Hom-J-dend alg3})$^t$=(Eq. \eqref{Hom-J-dend alg4}), (Eq. \eqref{Hom-J-dend alg4})$^t$=(Eq. \eqref{Hom-J-dend alg3}) and (Eq. \eqref{Hom-J-dend alg5})$^t$=(Eq. \eqref{Hom-J-dend alg5}).
                Hence $(A,\pt^t,\gr^t,\a)$ is a Hom-J-dendriform algebra.
                \end{proof}

\begin{prop} \label{hom-PJ==>hom-JD}
Let $(A,\c,\a)$ be a Hom-pre-Jordan algebra and $(V,l,r,\f)$ be a bimodule. Let $T:V \to A$ be an $\mathcal{O}$-operator of $A$ associated to $(V,l,r,\f)$.  Then there exists
a Hom-J-dendriform algebraic structure on $V$ given by
\begin{align}\label{J-dend bimodule}
    u \pt v=r(T(u))v,\ u \gr v=l(T(u))v, \quad \forall u,v \in V.
\end{align}
Therefore, there is a Hom-pre-Jordan algebra on $V$ given by Eq. \eqref{produit point}as the associated vertical Hom-pre-Jordan algebra of $(V,\pt,\gr)$ and $T$ is a homomorphism of Hom-pre-Jordan algebras. Moreover, $T(V ) = \{T(v) | v \in V\} \subset A $ is a Hom-pre-Jordan subalgebra of $(A, \c,\a)$, and there is an induced Hom-J-dendriform algebraic structure on $T(V )$ given by
\begin{align}
    T(u) \pt T(v)=T(u\pt v),\ T(u) \gr T(v)=T(u \gr v), \quad \forall u,v \in V.
\end{align}
Furthermore, its corresponding associated vertical Hom-pre-Jordan algebraic
structure on $T(V )$ is just the subalgebra of the Hom-pre-Jordan $(A, \c,\a)$, and $T$ is a
homomorphism of Hom-J-dendriform algebras.
\end{prop}

\begin{proof}
For any $a,b,c,u \in V$, we put
\begin{align*}
& T(a)=x,\ T(b)=y\  \textrm{and}\ T(c)=z.
\end{align*}
Then
\begin{align*}
   & \f(a \circ b)\gr \f (c \gr u)  \\
   =& (\f(a) \gr \f(b)+\f(b) \pt \f(a)+\f(b) \gr \f(a)+\f(a) \pt \f(b)) \gr (\f(c) \gr \f(u)) \\
   =& (l(T(\f(a)))\f(b)+r(T(\f(b)))\f(a)+ l(T(\f(b)))\f(a)+r(T(\f(a)))\f(b))\gr l(T(\f(c))) \f(u) \\
   =& l(T(l(T(\f(a)))\f(b)+r(T(\f(b)))\f(a)+ l(T(\f(b)))\f(a)+r(T(\f(a)))\f(b))) l(T(\f(c))) \f(u)  \\
   =& l(T(\f(a))\c T(\f(b)) +T(\f(b))\c T(\f(a)) )  l(T(\f(c))) \f(u)  \\
   =&  l(\a(x) \circ \a(y))l(\a(z)) \f(u)
\end{align*}
and
\begin{align*}
& \f^2(a) \gr [(b \circ c) \gr \f(u)] \\
=& \f^2(a) \gr [(b \gr c + c \pt b + c \gr b+ b \pt c) \gr \f(u)] \\
=& \f^2(a) \gr  l(T(l(T(b))c+r(T(c))b+l(T(c))b+r(T(b))c)) \f(u) \\
=&  \f^2(a) \gr l(y \circ z) \f(u) \\
=& l(T( \f^2(a)))   l(y \circ z) \f(u)   \\
=& l(\a^2(x)) l(y \circ z) \f(u).
\end{align*}
Hence
\begin{align*}
   & \f(a \circ b)\gr \f (c \gr u) +  \f(b \circ c)\gr \f (a \gr u) + \f(c \circ a)\gr \f (b \gr u)   \\
   =&  l(\a(x) \circ \a(y))l(\a(z)) \f(u) + l(\a(y) \circ \a(z))l(\a(x)) \f(u) + l(\a(z) \circ \a(x))l(\a(y))\f(u) \\
  =& l(\a^2(x))l(y\circ z)  \f(u) + l(\a^2(y))l(z\circ x)  \f (u) + l(\a^2(z))l(x\circ y)  \f (u) \\
  =&  \f^2(a) \gr [(b \circ c) \gr \f(u)] + \f^2(b) \gr [(c \circ a) \gr \f(u)] +
   \f^2(c) \gr [(a \circ b) \gr \f(u)] .
\end{align*}
Therefore, Eq. \eqref{Hom-J-dend alg1} holds.  Using a similar computation, Eqs. \eqref{Hom-J-dend alg2}--\eqref{Hom-J-dend alg5} hold.
Then $(V,\pt,\gr,\f)$  is a Hom-J-dendriform algebra.
The other conclusions  can be checked similarly.

\end{proof}

\begin{cor}\label{PJ==>JD by rotabaxter}
Let $(A, \c,\a)$ be a Hom-pre-Jordan algebra and let $R$ be a Rota-Baxter
operator $($of weight zero$)$ on $A$. Then the products, given by
$$x \pt y =y \c R(x), \  x \gr y= R(x) \c y, \quad \forall x,y \in A$$
define a Hom-J-dendriform algebra on $A$ with the same twist map.
\end{cor}

\begin{thm}
Let $(A, \c,\a)$ be a Hom-pre-Jordan algebra. Then there is a Hom-J-dendriform
algebra such that $(A, \c,\a)$ is the associated vertical
Hom-pre-Jordan algebra if and only if there exists an invertible $\mathcal{O}$-operator of $(A, \c,\a)$.
\end{thm}

\begin{proof}
Suppose that $(A,\pt,\gr,\a)$ is a Hom-J-dendriform algebra with respect to  $(A,\c,\a)$. Then  the identity map $id: A \to A$ is an $\mathcal{O}$-operator of $(A,\c,\a)$ associated to $(A,L_{\gr},L_{\pt},\a)$, where, for any $x,y \in A$
$$L_{\gr}(x)(y)=x \gr y\ \textrm{and}\ L_{\pt}(x)(y)=x \pt y.$$
Conversely, let $T: V \to A$ be an $\mathcal{O}$-operator of $(A,\c,\a)$ associated to a bimodule $(V,l,r,\f)$.  By Proposition \ref{hom-PJ==>hom-JD},  there exists a Hom-J-dendriform algebra  on $T(V)=A$ given by
$$T(u) \pt  T(v)=T(r(T(u)) v),\ T(u) \gr T(v)=T(l(T(u))v),\ \forall u,v \in V.$$
By setting $x=T(u)$ and $y=T(v)$, we get
$$x \pt y=T(r(x)T^{-1}(y))\ \textrm{and}\  x \gr y=T(l(x)T^{-1}(y)).$$
Finally, for any $x,y \in A$, we have
\begin{align*}
& x \gr y+ y \pt x=T(r(x)T^{-1}(y))  + T(l(x)T^{-1}(y))  \\
=& T(r(x)T^{-1}(y)  + l(x)T^{-1}(y))  =T(T^{-1}(x))\c  T(T^{-1}(y))=x \c y.
\end{align*}

\end{proof}

\begin{lem}\label{commuting rotabaxter}
Let $R_1$ and $R_2$ be two commuting Rota-Baxter operators $($of weight
zero$)$ on a Hom-Jordan algebra $(A,\circ,\a)$. Then $R_2$ is a Rota-Baxter operator $($of weight zero$)$ on the Hom-pre-Jordan algebra $(A,\c,\a)$, where
$$x \c y= R_1(x) \circ y, \quad \forall x,y \in A.$$
\end{lem}

\begin{proof}
For any $x,y \in A$, we have
\begin{align*}
& R_2(x) \c R_2(y)= R_1(R_2(x)) \o R_2(y) \\
=& R_2(R_1(R_2(x)) \o y + R_1(x) \o R_2(y)) \\
=& R_2(R_2(x) \c y + x \c R_2(y)).
\end{align*}
This finishes the proof.
\end{proof}

\begin{cor}
Let $R_1$ and $R_2$ be two commuting Rota-Baxter operators
$($of weight zero$)$ on a Hom-Jordan algebra $(A, \circ,\a)$. Then there exists a Hom-J-dendriform algebra structure on $A$ given by
\begin{align}\label{commuting rota-baxter}
    x \pt y= R_1(y) \circ R_2(x),\  x \gr y=R_1R_2(x) \circ y, \quad \forall x,y \in A.
\end{align}
\end{cor}

\begin{proof}
By Lemma \ref{commuting rotabaxter}, $R_2$ is Rota-Baxter operator of weight zero on $(A,\c,\a)$, where
$$x \c y= R_1(x) \o y.$$
Then, applying Corollary \ref{PJ==>JD by rotabaxter},   there exists a Hom-J-dendriform algebraic structure on $A$ given by
\begin{align*}
    x \pt y= R_1(y) \circ R_2(x),\  x \gr y=R_1R_2(x) \circ y, \quad \forall x,y \in A.
\end{align*}
\end{proof}

\end{document}